\documentclass[10pt]{article}
\usepackage{amsmath,amsthm,amssymb}
\usepackage{times}
\usepackage[active]{srcltx}
\usepackage{subeqnarray}
\usepackage{epsfig}
\usepackage{graphicx}      
\usepackage{subfigure}    
\usepackage{multido}      
\usepackage{pst-all}      
\usepackage{pst-poly}     

\newtheorem{theorem}{Theorem}[section]

\newtheorem{lemma}[theorem]{Lemma}

\newtheorem{example}[theorem]{Example}
\newtheorem{remark}[theorem]{Remark}

\def\qedsymbol{$\blacksquare$}
\def\eqn#1{\label{eq:#1}}
\def\eq#1{(\ref{eq:#1})} \def\Eq#1{Eq.~\eq{#1}} \def\Eqs#1{Eqs.~\eq{#1}}
\def\fig#1{\ref{fig:#1}} \def\Fig#1{Fig.~\fig{#1}} \def\Figs#1{Figs.~\fig{#1}}

\def\sup{\qopname\relax o{sup}}
\def\supp{\qopname\relax o{supp}}
\def\co{\qopname\relax o{co}}
\def\rank{\qopname\relax o{rank}}
\def\ep{\qopname\relax o{ep}}
\def\D{\qopname\relax o{D}}
\def\bound{\partial}
\def\-{\overline}
\def\c#1{{#1\llap{$\overline{\phantom{\mathrm{#1}}}$}}}
\def\cco{{\overline\co}}
\def\mc{\mathcal}

\def\mcq{\;,\quad}
\def\qh#1{\quad\hbox{#1}}
\def\qhq#1{\quad\hbox{#1}\quad}
\def\qqhqq#1{\qquad\hbox{#1}\qquad}
\def\lsup{\vee} \def\linf{\wedge}
\def\subsup{\triangledown} \def\subinf{\vartriangle}
\def\s{\sigma} \def\m{\mu} \def\n{\nu} \def\e{\epsilon} \def\j{\xi} \def\x{\chi}
\def\l{\lambda} \def\r{\rho}
\def\W{\Omega} \def\S{\Sigma}
\def\oo{\infty}
\def\Re{{\rm Re}}
\def\R{\mathrm{I\hskip-.5ex R}}
\def\N{\mathrm{I\hskip-.5ex N}}

\begin{document}

\title{The NMF problem and lattice-subspaces}

 \author{Ioannis A. Polyrakis \\
\\Department of Mathematics\\ National Technical University of Athens\\
ypoly@math.ntua.gr }
\maketitle

\begin{abstract} Suppose that $A$ is a nonnegative $n\times m$ real matrix.
The NMF problem is the determination of two nonnegative real matrices $F$,
$V$ so that $A=FV$ with  intermediate dimension $p$ smaller than $min\{
n,m\}$.
In this article  we present
 a general  mathematical method for the determination of two nonnegative real factors $F,V$ of $A$. During the  first steps of this process the intermediate dimension $p$ of $F,V$ is  determined, therefore we have an easy criterion for $p$.
   This
 study is based on the theory of lattice-subspaces and
positive bases. Also we give
the matlab program  for the computation of  $F,V$ but the mathematical part is the main part of this article.

\end{abstract}



\section{The NMF problem and positive bases}
Suppose that $A$
is a  nonnegative  $n\times m$   real matrix.
In  applications the
 nonnegative matrix factorization (NMF) problem   is  the following: Find  two  nonnegative real matrices $F$ and $V$ of  $n\times p$ and $p\times m$,  with $p\leq \min\{n,m\}$  such that
$A=FV$  or  the matrix $C=FV$ is an "approximation" of $A$ and  we say that the pair of factors $F,V$  is  an exact  NMF  or an  NMF approximation of $A$.
The NMF problem has  many applications in  data analysis  problems  such as in chemical concentrations,  document clustering, image processing, e.t.c. Since the exact NMF problem is not  solvable in general, the NMF problem  is commonly approximated numerically, see in ~\cite{Ber} and  ~\cite{Kim} for an introduction in numerical approximation methods.
Recall that the factorization    $A=FV$  implies  \begin{equation}\label{fact}a_i=\sum_{j=1}^pf_{ij}v_j,\;\text{and}\;a^i=\sum_{j=1}^p v_{ji}f^j,\end{equation} where for any  real matrix $B=(b_{ij})$,  denote by  $b_i$ the $i$-row  and by $b^i$ the $i$-column of $B$.
So each  row (column) of $A$ is a linear combination of the rows of $V$ (columns of $F$) and the coefficients of the $i$-row ($i$-column) of $A$ in this expansion  are the elements of the $i$-row of $F$ ($i$-column of $V$).
This implies  that \begin{equation}\label{fact0}\rank(A)\leq \min\{\rank(F),\rank(V)\}\leq p.\end{equation} The minimum   intermediate dimension  of all the factorizations  of $A$ is referred  in ~\cite{Cohen}  as  {\it the nonnegative rank}  of $A$ and it is denoted by $\rank_+(A)$. This implies that  $$rank(A)\leq rank_+(A)\leq \min\{n,m\}.$$
If  the intermediate dimension $p$ of the factors  $F,V$ is equal  to the rank of $A$, we say that $F,V$ is a {\it nonnegative rank factorization} (NRF) of $A$, see in ~\cite{Camp}. The NRF problem does not have always a solution because $\rank_+(A)$ may be strictly greater than $\rank(A)$ but the   determination of $\rank_+(A)$  of $A$ is an interesting problem of matrix factorization. The factorization $A=FV$ is a {\it symmetric}  NRF  of $A$ if it is a nonnegative rang factorization of $A$ with  $V=F^T$, i.e. $A=FF^T$, where $F^T$ is  the transpose of $F$.\\
In this article we give a general mathematical method based on the theory of lattice-subspaces and positive bases expanded in ~\cite{POLY96} and ~\cite{POLY99} which determines an   exact  factorization of $A$ in the nonnegative factors $F,V$ without any restriction for $A$.
During the first steps of our algorithmic process and before the determination of the factors $F,V$  we can determine the intermediate dimension $p$ of  $F,V$, and therefore  we can  know if $F,V$ will be  an NRF,  an NMF,  a trivial, or an other kind factorization of $A$.\\
 For the determination of the factors $F,V$ of $A$   we
determine  a positive basis of a minimal  lattice-subspace $Z$ of
$\mathbb{R}^m$ which contains  a maximal set of linearly independent rows  of $A$. The elements of the positive basis fixe some of the
columns of $A$ and also indicate some real numbers so that   $F$ is the matrix generated by positive multiples of these columns of $A$ by the corresponding  real numbers and $V$ is the matrix with rows the vectors of the positive basis of $Z$.
Of course an exact NMF factorization of $A$ is not always possible because $p<\min\{n,m\}$ requires    $\rank(A)< \min\{n,m\}$.
In  the case where $\rank(A)=2$ we determine  by a simple and very easy way an exact rank  factorization of $A$,  see in Subsection~\ref{rank}.  This result is not new.  The case of $\rank(A)=2$   has been studied    in ~\cite{Cohen} and ~\cite{Amb} where  an algorithmic process for an exact NRF of $A$ is proposed but our way  is very simple and a part of our general method of matrix factorization.\\
In  ~\cite{Kal}, an exact, symmetric nonnegative rank factorization of $A$, i.e. $A=WW^T$, is determined in the case where $A$ is a  symmetric $n\times n$ nonnegative real matrix  which contains a diagonal principal submatrix of the same rank with $A$.  In  the more general case where $A$ is a $n\times m$ nonnegative real matrix which contains a diagonal principal submatrix of the same rank with $A$, we show that by our  factorization method we take also a nonnegative rank factorization of $A$, Subsection~\ref{diag}, but if $A$ is symmetric this  factorization is not necessarily symmetric as in  ~\cite{Kal}.

In this article, in order to simplify  computations,  we suppose that  $A$ does not have zero columns. Indeed, if before the determination of the factors  $F,V$ we have deleted the zero columns of $A$ then if we put  zero columns in $V$   in the place of the deleted columns of $A$, the product $FV$ is the initial matrix $A$.
In our matlab program we delete the zero columns of $A$ by the function "Zero" and after the determination of $F,V$   we add  zero columns in $V$, if it is needed,  by the function "addzeros".

So in this article  we suppose that $A=(a_{ij})$ is a  nonnegative $n\times m
$  real matrix without zero columns.  For any $i$,
$a_i=(a_{i1},a_{i2},...,a_{im})$ is the $i$-row of $A$ and we will also denote by $a_i(j)$ the $j$-coordinate of $a_i$, i.e. $a_i(j)=a_{ij}$.
We start by a maximal set
 $$\{y_{1},y_{2},...,y_{r}\}$$      of
 linearly independent rows of $A$  which we will  refer  as    {\bf basic set}  (of the rows of $A$) and
   we  will denote by $X$ the subspace of $\mathbf{R}^m$ generated by these vectors, i.e. $$X=[y_{1},y_{2},...,y_{r}].$$ Then  $a_i\in X$ for any $i$.
 In the sequel we determine a positive basis $\{b_1,b_2,...,b_d\}$ of a minimal lattice-subspace
 $Z$ of $\mathbf{R}^m$ which contains the vectors  $\{y_{1},y_{2},...,y_{r}\}$  and  we state Theorem~\ref{1} which gives
 a method for the determination of $F,V$. {\it The fact that  $\{b_1,b_2,...,b_d\}$ is a positive basis of a minimal lattice-subspace $Z$ which contains the basic set $\{y_{1},y_{2},...,y_{r}\}$  is crucial because it ensures that  any $y_i$ and also any other row of $A$ has nonnegative coordinates in this basis.}

 For the determination of a  positive basis
$\{b_1,b_2,...,b_d\}$ of $Z$  we follow the steps of ~\cite{POLY96}, Theorem 3.7 and ~\cite{POLY99}, Theorem 3.10 (Theorem ~\ref{Prop4} and Theorem ~\ref{Prop5} in the Appendix).
 We describe below the process for the determination of a positive basis of $Z$ because  it is needed  in the proof of Theorem~\ref{1}, in the examples  and also in the whole article.

 In the first step we start by a fixed basic set $\{y_{1},y_{2},...,y_{r}\}$ and we determine  the  {\bf basic function} $\beta :\{1,2,...,m\}\longrightarrow \mathbb{R}^r_+$ of the vectors $y_i$. This  function has been defined in   ~\cite{POLY96} and is the following\footnote{In ~\cite{POLY96} the basic function is referred as basic curve}:
$$ \beta(i)=\Bigl(\frac{y_1(i)}{y(i)},\frac{y_2(i)}{y(i)},...,
\frac{y_r(i)}{y(i)}\Bigr),\;\text{for each}\;\;
i=1,2,...,m\;\text{with\;} y(i)>0,$$  where $y=y_1+y_2+...+y_r$,
is  the sum of the  vectors $y_i$. In our case we have $y(i)>0$ for any $i$ because we have deleted the zero columns of $A$. \footnote{Note that  $y(i)=||(y_1(i),y_2(i),...,y_r(i))||_1$, where $||.||_1$ is the $\ell_1$-norm of $\mathbf{R}^r$ and this notation  is used in ~\cite{POLY96}.}
The  function $\beta$ takes values in the simplex
$\Delta=\{x\in\mathbb{R}^r_+\;|\;\sum_{i=1}^rx(i)\;=1\} $ of the
positive cone  of $\mathbb{R}^r$. The set
 $$R(\beta)= \{\beta(i)\;|\;i=1,2,...,m\},$$ is the range of $ \beta$. Of  course   $R(\beta)$, as a set,  is consisting by mutually different vectors and also
 $R(\beta)$  contains exactly $r$ linearly independent vectors,   see Lemma~\ref{l1} in the Appendix.
Suppose that $\mu$ is  the cardinal number  of $R(\beta)$, i.e. $\mu$ is  the number of the different values of the basic function $\beta$. Then $$r\leq \mu \leq m.$$
 In the sequel, according to Theorem~\ref{Prop5} in the Appendix  we consider  the convex polytope $K$ of $\Delta$ generated by the vectors of $R(\beta)$ and suppose that $P_1,P_2,...,P_d $ are the vertices of $K$. Then   $\beta(i)\in K$ for any $i$ and any  $\beta(i)$ is a convex combination of the vertices $P_i$. This  shows    that the set of   vertices of $K$ contains a maximal set of linearly independent vectors of $R(\beta)$ and  according to Lemma~\ref{l1} this set  has exactly $r$ elements. So we have that $$r\leq d\leq \mu\leq m.$$
According to Theorem~\ref{Prop5},
 we reenumerate the vertices $P_i$  of $K$ so that the first $r$ of theme to be  linearly independent and we denote again by \begin{equation}\label{eq1} P_1,P_2,...,P_d \end{equation}  the new enumeration of the vertices  where the first $r$ of theme  are linearly independent.
In  Theorem~\ref{1}, we show that the intermediate dimension $p$ of the factors $F,V$ is equal to the number $d$ of vertices of $K$. If
 $d\geq \min\{n,m\}$  our matlab programm sent the sign {\it the intermediate dimension is equal to M1} and also $M1$ is appeared (in  matlab program $M1=d$). So we have the possibility to continue or not.

If $r=d$, then according to Theorem~\ref{Prop4}, $X$ is a lattice-subspace and therefore $X$ is the minimal lattice subspace which contains the vectors $y_i$, i.e.  $Z=X$ and a positive basis $\{b_1,b_2,...,b_d\}$ of $Z$ is given by the formula $$(b_1,b_2,...,b_d)^T=L^{-1}(y_1,y_2,...,y_d)^T,$$
where $L$ is the matrix with columns the vectors of $P_1,P_2,...,P_d$ and in the sequel, the factors $F,V$ are   determined by Theorem~\ref{1}.
 In the  case where  $\mu=r$, i.e. if $R(\beta)$ has exactly $r$ elements,  the vertices of $K$ are the vectors  of $R(\beta)$ therefore
   we have again that $d=r$, the basis $\{b_1,b_2,...,b_d\}$ of $Z$ is given by the above formula so the factors $F,V$  by Theorem~\ref{1}.
 So we determine $F,V$ by avoiding the   computations of the  vertices  of $K$. Note also that if $r=\mu$,
$X$ is a sublattice of $\mathbb{R}^m$ although  this information is not important for our method. In both of these cases, we have $\rank(A)=d$ therefore $A=FV$ is a rank factorization of $A$ and by our matlab program we determine this factorization with the sign
 {\it Rank factorization, the rows of the matrix generate a lattice-subspace(sublattice)}.\\

 If $r<d\leq m$, according to Theorem ~\ref{Prop5} in the Appendix we define  $d-r$ new vectors $y_{r+1},...,y_d$ of $\mathbb{R}^m$ following the next steps:\\
First, for any $i=1,2,...m$, we expand  $\beta(i)$ as a convex combination  of the vertices $P_1,P_2,...,P_d$ and suppose that   \begin{equation}\label{eq0} \beta(i)=\sum_{j=1}^d\xi_j(i)P_j.\end{equation}  Such an expansion of $\beta(i)$ is not necessarily unique but we select one of theme.  Of course $\xi_j(i)\geq 0$ for any $j$  and $\sum_{j=1}^d\xi_j(i)=1$.
In the sequel for any $k=r+1,...,d$, we define the  new vector $y_{k}$  of $\mathbb{R}^m$ as follows: $$y_{k}(i)=\xi_k(i)y(i)\;\text{for any}\;i=1,2,...,m,$$ where $y$ is the sum of the vectors $y_1,...,y_r$.
According to Theorem~\ref{Prop5} in the Appendix, the   subspace $$Z=[y_1,...,y_r,y_{r+1},...,y_d],$$ generated by these vectors is a $d$-dimensional  minimal lattice-subspace of $\mathbb{R}^m$ which contains the vectors $y_1,...,y_r$.
  So a positive basis of $Z$ is determined  by Theorem~\ref{Prop4} in the Appendix, because the subspace $Z$  generated by these vectors is a lattice-subspace. Therefore for the determination of a positive basis of $Z$, according to Theorem~\ref{Prop4},   we take the  basic function of the vectors $y_1,...,y_r,y_{r+1},...y_d$ which we denote by $\gamma$, i.e. the function
\begin{equation}\label{eq2} \gamma(i)=\Bigl(\frac{y_1(i)}{y'(i)},\frac{y_2(i)}{y'(i)},...,
\frac{y_d(i)}{y'(i)}\Bigr),\;\text{for each}\;\;
i=1,2,...,m,\end{equation}  where $y'=y_1+y_2+...+y_d$,
is  the sum of the  vectors $y_i,\; i=1,2,...d$. Then  by Theorem~\ref{Prop4}, the convex hull of the values $\gamma(i)$, $i=1,2,...,m$ of $\gamma$ is a convex polytope with  vertices the linearly independent vectors $R_1,R_2,...,R_d$ of $\mathbb{R}^m$ and a positive basis $\{b_1,b_2,...,b_d\}$ of $Z$ is given by the formula $$(b_1,b_2,...,b_d)^T=L^{-1}(y_1,y_2,...,y_d)^T,$$
where $L$ is the matrix with columns the vectors $R_1,R_2,...,R_d$. Of course any  $R_k$ is the image $\gamma(i_k)$ of an  index $i_k$ because the convex polytope is generated by a finite number of vectors. According to Theorem~\ref{Prop5}, the above vectors $R_i$ can be determined  by the vectors $P_1,P_2,...,P_d $ of (\ref{eq1})  where the first $r$ of theme are  linearly independent as follows: $R_i=(P_i,0)$, for any $i=1,2,...,r$ and $R_{r+k}=(P_{r+k},e_k)$, for any $k=1,2,...,d-r$ where
the second component of  the vectors $(P_i,0)$, is the  the zero vector of $\mathbf{R}^{d-r}$ and the second component $e_k$ of $(P_{n+k},e_k)$ is the $k$-vector of the usual basis $\{e_1,...e_{d-r}\}$  of $\mathbf{R}^{d-r}$. This way of the determination of  $R_i$ simplifies the  computations because  avoids the determination of the vertices of the convex polytope generated by  $R(\gamma)$ and we adopt  this  way in our the matlab program.\\
Note that if $d=m$, Theorem~\ref{1} gives   a trivial factorization of $A$.
This can occur if $r=m$ or if $r<d=m$. In both cases   the matlab program is interrupted with the sign {\it  trivial factorization}.

We recall now the definition of the positive basis with nodes.
Suppose that   $Y$ is a  lattice-subspace  of $\mathbb{R}^m$ with a positive basis
  $\{b_1,b_2,...,b_\nu \}$.
If for some $k\in \{1,2,...,\nu\}$ there exists an   index $i_k$ so that
 $b_k(i_k)>0$ and $b_j(i_k)=0$ for any $j\neq k$, we say that $i_k$ is a $k$-node of the basis $\{b_1,b_2,...,b_\nu\}$.
If a set $\{i_1,i_2,...,i_\nu\}$ of indices exists so that for any $k=1,2,...,\nu$, the index $i_k$ is a $k$-node of the basis, we say that  $\{i_1,i_2,...,i_\nu\}$ is a  {\it set of  nodes}  of the basis $\{b_1,b_2,...,b_\nu\}$ and also that $\{b_1,b_2,...,b_\nu\}$ is a  {\it basis  with nodes}. Then  for any $k=1,2,...,\nu$ we have $b_k(i_k)>0$ and $b_j(i_k)=0$  for any $j=1,2,...,\nu$ with $j\neq k$.

Before to state our factorization theorem  we note  that  according to  ~\cite{POLY99}, Example 3.21, a minimal lattice-subspace $Z$ of  $\mathbb{R}^m$ which contains the basic set  $\{y_1,...,y_r\}$  is not necessarily unique. This seems also by  Example~\ref{ex1} of this article where two different minimal lattice-subspaces are determined. The reason  is due to the fact that in (\ref{eq0}) the convex combination of the values $\beta(i)$  and therefore also and   the new  vectors $y_{r+1},...,y_d$ are  not necessarily uniquely determined.

For the sake of completeness and for the importance of the next result we give its proof. The proof  can be also followed   by ~\cite{POLY96} in  the  special case where  $\Omega=\{1,2,...,m\}$ but not directly. In the next result we keep  the above terminology.

\begin{theorem}\label{0} If $Z$ is a minimal  lattice-subspace  of $\mathbb{R}^m$  which contains  the linearly independent  vectors   $y_1,y_2,...,y_r$ of $\mathbb{R}^m_+$ and $\{b_1,b_2,...,b_d\}$ is a positive basis of $Z$, then $\{b_1,b_2,...,b_d\}$ is a basis with nodes.
\end{theorem}

\begin{proof} The vectors of a  positive basis of $Z$  are  unique in the sense of a positive multiple and a permutation. So we may suppose that $\{b_1,b_2,...,b_d\}$ is a positive basis of $Z$ which arises by the above process where of course $d\geq r$. As we have seen above, $d$ is the number of the vertices of the convex polytope $K$ generated by  $R(\beta)$.

First we consider the case $d>r$. Then according to
 Theorem ~\ref{Prop5}, the basis  is given by the formula
 \begin{equation}\label{e3}(b_1,b_2,...,b_d)^T=L^{-1}(y_1,y_2,...,y_d)^T,
 \end{equation}
 where $L$ is the matrix with columns the vertices  $R_1,R_2,...,R_d$ of the convex polytope generated by the values $\gamma(i)$ of the basic function $\gamma$ of the vectors $y_1,y_2,...,y_d$.
So we have $$L(b_1,b_2,...,b_d)^T=(y_1,y_2,...,y_d)^T.$$  Then any $R_k$ is the image $\gamma(i_k)$ of an index $i_k$ and so  we take a set of indexes $\{i_1,i_2,...,i_d\}$, with  $R_k=\gamma(i_k)$ for any $k$.
 Therefore we have
$$(R_k)^T=(\gamma(i_k))^T=\frac{1}{y'(i_k)}(y_1(i_k),y_2(i_k),...,y_d(i_k))^T=\frac{1}{y'(i_k)}L(b_1(i_k),b_2(i_k),...,b_d(i_k))^T.$$
Since $L$ is the matrix with columns the vectors $R_{1},R_{2},...,R_{d}$ we have
$$ (R_k)^T=\sum_{j=1}^d\frac{b_j(i_k)}{y'(i_k)}(R_j)^T,$$ therefore $$\frac{b_k(i_k)}{y'(i_k)}=1\;\text{and}\;\frac{b_j(i_k)}{y'(i_k)}=0,\;\text{ for any}\; j\neq k,$$  because the vectors $R_i$ are linearly independent.  Therefore
$b_k(i_k)>0$ and $b_j(i_k)=0$ for any $j\neq k$  and      $\{i_1,i_2,...,i_d\}$ is a set of nodes.

In the case  where $d=r$,  according to Theorem~\ref{Prop4}, the basis of $Z$ is given again  by (\ref{e3}) where  $L$ is the matrix with columns the vertices  $P_1,...,P_d$ of $K$. We repeat the above proof where  in the place of $R_i$ we have the vectors $P_i$ and  we  find again a set of nodes $\{i_1,i_2,...,i_d\}$  of the basis $\{b_1,b_2,...,b_d\}$.
\end{proof}

In the next result we give a factorization of $A$.  As we have also noted before, the intermediate dimension $p$ of the factors $F,V$ is equal to the number $d$ of vertices of $K$ and is determined before the  determination of the positive basis of $Z$.
Underline  the case where $\rank(A)=2$ because then  the intermediate dimension of the factors is equal to $2$.
\begin{theorem}\label{1}({\bf Matrix factorization}) Suppose that $A$ is a nonnegative $n\times m$ real matrix, without zero columns, $a_i=(a_{i1},a_{i2},...,a_{im})$, $i=1,2,...n,$ are the rows of $A$,
   $\{y_{1},y_{2},...,y_{r}\}$ is a basic set of the rows of
 $A$, $\beta$ is the basic function of $y_{1},y_{2},...,y_{r}$,  $K$ is the convex polytope generated by the range $R(\beta)$ of $\beta$ and $d$ is the number of the vertices of $K$. Then there exist nonnegative   real matrices $F$, $V$ of intermediate dimension $d$ so that \begin{equation}\label{f}A=FV\end{equation} which are determined as follows:\\
We determine a positive basis $\{b_1,b_2,...,b_d\}$ of  a  minimal lattice-subspace $Z$  of $\mathbb{R}^m$ which contains the vectors
$y_{1},y_{2},...,y_{r}$ and a  set of nodes $\{i_1,i_2,...,i_d\}$ of this basis. Then
  $F$ is the $n\times d$ matrix so that for any $k=1,2,...,d$  the $k$-column of $F$ is the  $i_k$-column of $A$ multiplied by $\frac{1}{b_k
(i_k)}$  and $V$ is the  $ d\times m$ matrix with rows the vectors $b_1,b_2,...,b_d$.\\
For this factorization of $A$ we discriminate the cases:

$(i)$ If r=2, then $d=2$  and (\ref{f}) is an exact rank factorization of $A$.

$(ii)$ If  the function $\beta$ takes exactly $r$ different values,  then (\ref{f}) is a rank factorization of $A$.

$(iii)$ If $K$ has $r$ vertices, then $r=d$  and (\ref{f}) is a rank factorization of $A$.

$(iv)$ If  $d=m$,  (\ref{f})  is a trivial the factorization of $A$ (i.e. the set of rows of  $V$ is a positive basis of $\mathbf{R}^m$).

\end{theorem}

\begin{proof} Let $Z$ be a minimal lattice-subspace   of $\mathbb{R}^m$ which contains the vectors
$y_{1},y_{2},...,y_{r}$ and suppose that $$\{b_1,b_2,...,b_d\}$$
is a positive basis of $Z$ constructed by the method described before the theorem.
Then for any row $a_i$ of $A$, $a_i$ belongs to $Z$ because $Z$ is a subspace of
$\mathbb{R}^m$ containing a maximal set of linearly independent rows of $A$ and suppose that
$$a_i=\sum_{j=1}^d f_{ij}b_j.$$
Therefore, by (~\ref{fact}),  we have $$A=FV,$$ where $F=(f_{ij})$ is the $n\times d$ matrix of the coefficients of the vectors $a_i$ in the basis $\{b_1,b_2,...,b_d\}$ and $V$ is the matrix with rows the vectors $b_i$. Of course we have $f_{ij}\geq 0$ for any $i,j$ because $\{b_1,b_2,...,b_d\}$ is a positive basis of $Z$ and $a_i\in Z_+$ for any $i$.
We determine now the coefficients $f_{ij}$  as follows: We take a set of nodes  $\{i_1,i_2,...,i_d\}$   of the basis $\{b_1,b_2,...,b_d\}$. For any fixed $k\in\{
 1,2,...,d\}$   we have
 $b_k(i_k)>0$ and $b_j(i_k)=0$ for any $j\neq k$. Therefore for any $i$ we have
$$a_i(i_k)=\sum_{j=1}^d f_{ij}b_j(i_k)=f_{ik}b_k(i_k)$$ hence  $$f_{ik}=\frac{a_i(i_k)}{b_k(i_k)}=\frac{a_{ii_k}}{b_k(i_k)},\;\text{for any }\;i=1,2,...,n.$$
Therefore,  the $k$-column of $F$ is the  $i_k$-column of $A$ multiplied by $\frac{1}{b_k
(i_k)}$.

$(i)$ Suppose that $r=2$. Then $\{y_1,y_2\}$ is a basic set of the rows of $A$. We shall show that $d=2$. The basic function $\beta$ of $y_1,y_2$, takes values on the one-dimensional simplex $\Delta$ of
$\mathbf{R}^2_+$ defined by the points $(1,0)$ and $(0,1)$ of $\mathbf{R}^2$.
Therefore the convex polytope $K$ generated  by the values  of $\beta$ is a
line segment defined by two values of $\beta$, the ones with
minimum and maximum first coordinate. Therefore $K$ has two vertices hence  $d=2$. So  $F,V$ are  $n\times 2$, $2\times m$ matrices and (~\ref{f}) is a rank factorization of $A$.

$(ii)$ Suppose that $\beta$ has $r$ different values. Then by Lemma ~\ref{l1}, the values of $\beta$ are linearly indpendent, therefore the convex polytope generated by the valyes of $\beta $ has $r$ vertices. So $d=r$ and  (~\ref{f}) is a rank factorization of $A$.

$(iii)$ If $d=r$, then (~\ref{f}) is again  a rank factorization of $A$.

$(ii)$ If $d=m$, then   $Z=\mathbf{R}^m$ because $Z$ id a $d$-dimensional subspace of $\mathbf{R}^m$. So $\{b_1,b_2,...,b_d\}$ is a  positive basis  of $\mathbf{R}^m$ and the factorization (~\ref{f}) is trivial.
\end{proof}

\begin{remark} {\rm In  the above theorem we have defined the factorization $A=FV$ of $A$ as   {\it trivial } if the set of rows of  $V$ is a positive basis of $\mathbf{R}^m$. Recall that a  positive basis  of $\mathbf{R}^m$ is unique in the sense of positive multiples and  permutations. So, if  the rows of $V$ are the vectors of the positive basis $\{e_1,...,e_m\}$, then $\{1,...,m\}$ is the set of nodes of the basis and according to Theorem~\ref{1} the columns of $F$ are the columns of $A$ multiplied by $1$, therefore  $F=A$ and $V=I_m$ is the identical matrix. In the case of  an other positive basis  of $\mathbf{R}^m$,  the columns of $F$ are again positive multiples of the columns of $A$  but maybe  in an other  order.}
\end{remark}

\subsection{The case $\rank(A)=2$}\label{rank} We discuss the case $r=2$ or equivalently $\rank(A)=2$. This case has been studied    in ~\cite{Cohen} and ~\cite{Amb} where  an algorithmic process  has been proposed  for the   determination of a nonnegative rank factorization of $A$.
We present here, as a partial part of the Factorization Theorem,  a  simple and very easy way, for the determination of a NRF of $A$.

  If  $r=2$,  we start by a basic set $\{y_1,y_2\}$ of the rows of $A$ and we take the basic function  $\beta$ of $y_1,y_2$. The vales of $\beta$ are on the one-dimensional simplex $\Delta$ of
$\mathbf{R}^2_+$ defined by the points $(1,0)$ and $(0,1)$, therefore
 the convex polytope $K$ generated  by $R(\beta)$ is the
line segment defined by two values of $\beta$, the ones with
minimum and maximum first coordinate. So if $a$ is the minimum and $b$ is the maximum first coordinate of the values of $\beta$ then $K$ is the line segment defined by the points $P_1=(a,1-a)$ and $P_2=(b,1-b)$ and these points are the vertices of $K$.  Therefore a positive basis of the minimal lattice-subspace $Z$ which contains $y_1,y_2$, is given by the formula
 $$(b_1,b_2)^T=L^{-1}(y_1,y_2)^T,$$ where $L$ is the $2\times 2 $ matrix with columns the vectors $P_1$ and $P_2$  and
 it is very easy to find a set of nodes $\{i_1,i_2\}$ of the basis. Then $F$ is the $n\times 2$ matrix with first column the $i_1$-column of $A$ multiplied by $\frac{1}{b_1(i_1)}$ and the second column of $F$ is the $i_2$-column of $A$ multiplied by $\frac{1}{b_2(i_2)}$ and $V$ is the $2\times m$ matrix with rows the vectors $b_1,b_2$ of the positive basis of $Z$. For an application see Example ~\ref{ex3}.

\section{Examples}
 In the next examples the matrix   $A$ is  without zero columns.

\begin{example}\label{ex1}\rm{ In this example we  show    the way the algorithmic process  is working. Also    two different minimal lattice-subspaces which contain the vectors  $y_i$ are appeared.
Suppose that
$$ A=\begin{bmatrix}
               1&2&1&2&0&0\\
               0&3&0&4&2&1\\
               0&3&1&5&2&0\\
               1&5&1&6&2&1\\
               1&5&2&7&2&0\\
               0&6&1&9&4&1
               \end{bmatrix}.$$
We find that $\{y_1=a_1,y_2=a_2,y_3=a_3\}$ is a basic set, i.e.
a maximal set of linearly independent rows of $A$.
  The basic function  of the vectors
$y_i$ is $$\beta(i) =\frac{1}{y(i)}(y_1(i),y_2(i),y_3(i)),\;i=1,...,6,$$ where
$y=(1, 8, 2, 11, 4, 1)$ is the sum of the vectors $y_i$.  The values of $\beta$ are on the simplex $\Delta$ of $\mathbf{R}^3_+$ and we have:
 $\beta(1)=(1,0,0)$,
$\beta(2)=\frac{1}{8}(2,3,3)$,
$\beta(3)=\frac{1}{2}(1,0,1)$,
$\beta(4)=\frac{1}{11}(2,4,5)$,
 $\beta(5)=\frac{1}{2}(0,1,1)$ and
$\beta(6)=(0,1,0)$ and
 $$R(\beta)=\{\beta(1),\beta(2),\beta(3),\beta(4),\beta(5),\beta(6)\}$$ is the range of $\beta$.
We determine  the vertices of the convex polytope  $K$ generated by $R(\beta)$ and we
  find  that
   $P_1=\beta(1),P_2=\beta(6),P_3=\beta(5),P_4=\beta(3)$ are the vertices of $K$. Therefore   $d=4$  and the intermediate dimension of the factors $F,V$  will be equal to $4$.

   We take a new
   enumeration of the vertices of $K$, which we denote again  by $P_1,P_2,P_3,P_4$   so that the first three  of them ($r=3$) to be
   linearly independent. Of course there are four  such enumerations and the above  is one of theme but for compatibility, we
   adopt the next one of our matlab program:
 $$P_1=\frac{1}{2}(0,1,1)=\beta(5), P_2=(0,1,0)=\beta(6),P_3=\frac{1}{2}(1,0,1)=\beta(3),P_4= (1,0,0)=\beta(1).$$
We expand the values of $\beta$ as convex combinations of the vertices of $K$ i.e.
 $$\beta(i)=\sum_{j=1}^4\xi_j(i)P_j,$$ for any $i$, where $\xi_j(i)\geq 0$ and $\sum_{j=1}^4\xi_j(i)=1$.
Except of the case where $\beta(i)$ is a vertex of $K$, this
expansion is not necessarily unique.
From this point we continue by two ways. In the first one,   we follow the
steps of  the algorithm without the use of a computer and in the
second  we follow the computations of  the matlab program.\\
So we have:
 $\beta(1)=P_4$, therefore $\beta(1)=\sum_{i=1}^4\xi_i(1)P_i,$ with $\xi_i(1)=0$ for $i=1,2,3$ and $\xi_4(1)=1$. $\beta(2)$ is not a vertex. We find that
   $\beta(2)=\sum_{i=1}^4\xi_i(2)P_i$, and it is easy that    $\xi_1(2)=\frac{3}{4},$  $\xi_2(2)=\xi_2(2)=0$ and  $\xi_4(2)=\frac{1}{4}$. $\beta(3)=P_3$, therefore  $\xi_i(3)=0$ for $i=1,2,4$ and $\xi_3(3)=1$. $\beta(4)$ is not a vertex and we find that $\xi_1(4)=\frac{6}{11}$, $\xi_2(4)=\frac{1}{11}$, $\xi_3(4)=\frac{4}{11}$, $\xi_4(4)=0$. $\beta(5)=P_1$, so $\xi_1(5)=1$ and  $\xi_i(5)=0$  for $i=2,3,4$ and $\beta(6)=P_2$, so $\xi_2(6)=1$ and  $\xi_i(6)=0$  for $i=1,3,4$. The next is the  matrix with rows the coefficients of $\beta(i)$
 $$ H2=\begin{bmatrix}
               0&0&0&1\\
               \frac{3}{4}&0&0&\frac{1}{4}\\
               0&0&1&0\\
               \frac{6}{11}&\frac{1}{11}&\frac{4}{11}&0\\
               1&0&0&0\\
               0&1&0&0
               \end{bmatrix}.$$
According to the algorithm we define $d-r$ new vectors, therefore
 we define one new vector $y_4$ of $\mathbb{R}^6$  as
follows: $y_4(i)=\xi_4(i)y(i)$ for any $i$, where
$y=y_1+y_2+y_3=(1,8,2,11,4,1)$, therefore $$y_4=(1,2,0,0,0,0).$$
 Then  $Z=[y_1,y_2,y_3,y_4]$ is  a minimal lattice-subspace  containing $y_1,y_2,y_3,y_4$ and a
 positive basis $\{b_1,b_2,b_3,b_4\}$ of  $Z$
 is given by the formula
 $$(b_1,b_2,b_3,b_4)^T=L^{-1}(y_1,y_2,y_3,y_4)^T,$$  where $L$ is the matrix with rows the vectors
 $R_1,R_2,R_3,R_4$ and the the vectors $R_i$ are the following:
 $R_1=(P_1,0)=\frac{1}{2}(0,1,1,0)$, $R_2=(P_2,0)=(0,1,0,0)$, $R_3=(P_3,0)=\frac{1}{2}(1,0,1,0)$,
 $R_4= \frac{1}{2}(P_4,e_1)=\frac{1}{2}(1,0,0,1)$. We find that
$$b_1=( 0,    6,     0,     6,     4,     0), b_2=(0,     0,     0,     1,     0,     1),
    b_3=( 0,     0,     2,     4,     0,     0), b_4=( 2,     4,     0,     0,    0,     0),$$
 is the positive basis of $Z$.
The set of indexes     $\{i_1=5,i_2=6,i_3=3,i_4=1\}$ is a  set of
nodes of the basis. Therefore  $A=FV$ where $V$ is the matrix with
rows the vectors $b_i$ and $F$ is the $6\times 4$ matrix so that
the first column of $F$ is the fifth column of $A$ multiplied by
$\frac{1}{b_1(i_1)}=\frac{1}{4}$, the second  column of $F$ is the
sixth column of $A$ multiplied by $\frac{1}{b_2(i_2)}=1$, the
third column of $F$ is the third column of $A$ multiplied by
$\frac{1}{b_3(i_3)}=\frac{1}{2}$ and the fourth  column of $F$ is
the first column of $A$ multiplied by
$\frac{1}{b_4(i_4)}=\frac{1}{2}$. So we find that

$$F=\begin{bmatrix}
               0&0&\frac{1}{2}&\frac{1}{2}\\
               \frac{1}{2}&1&0&0\\
               \frac{1}{2}&0&\frac{1}{2}&0\\
               \frac{1}{2}&1&\frac{1}{2}&\frac{1}{2}\\
              \frac{1}{2}&0&1&\frac{1}{2}\\
               1&1&\frac{1}{2}&0
               \end{bmatrix},\;\;V=\begin{bmatrix}0&    6&    0&     6&     4&     0\\ 0&     0&     0&     1&     0&     1\\
          0&     0&     2&     4&     0&     0\\  2&     4&     0&     0&    0&     0\end{bmatrix}$$
         and it is easy to check that  $FV=A$.

  In the second way, following  our matlab computations we obtain  as above the same values of $\beta$ and the convex polytope $K$. We find that

      $$ H2=\begin{bmatrix}
               0&0&0&1\\
               0.4&0.175&0.35&0.075\\
               0&0&1&0\\
               0.5572&0.0851&0.3519&0.0059\\
               1&0&0&0\\
               0&1&0&0
               \end{bmatrix},$$
is a matrix with rows  convex combinations of the values  of $\beta$ in the vertices of $K$ and we find that
     $$y_4=(  1,    0.6,         0,    0.0644,         0,         0),$$ is the new
     vector. We remark  that     we
     have found different convex combinations for $\beta(2)$ and $\beta(4)$ and also that $y_4$ is different than the previous one.  $Z=[y_1,y_2,y_3,y_4]$ is a minimal lattice-subspace  containing $y_1,y_2,y_3$ but different from  the one of the fist case because  $y_4$ is not    a linear combination of the vectors of the positive basis of $Z$ of the first case.
     According to our matlab program we find that the rows of the matrix
     $$U= \begin{bmatrix}0&   3.2000&         0&    6.1288&    4.0000&         0\\
         0&    1.4000&         0&    0.9356&         0&    1.0000\\
         0&    2.8000&    2.0000&    3.8712&        0&         0\\
    2.0000&    1.2000&         0&    0.1288&         0&         0\end{bmatrix},$$
     are the  the vectors of a positive basis $\{b_1,b_2,b_3,b_4\}$ of  $Z$  and also that
   $\{i_1=5,i_2=6,i_3=3,i_4=1\}$ is a set of nodes of this basis. We take the factors
 $$   F =\begin{bmatrix}
         0&         0&   0.5000&    0.5000\\
    0.5000&    1.0000&         0&0\\
    0.5000&         0&   0.5000& 0\\
    0.5000&    1.0000&   0.5000&    0.5000\\
    0.5000&         0&    1.0000&   0.5000\\
    1.0000&    1.0000&    0.5000&         0\end{bmatrix},\; V=U$$
 of $A$ and the test matrix  $R=A-FV=0$.}

\end{example}

\begin{example}\label{ex2}\rm{Suppose that

$$ A=\begin{bmatrix}
    1&2&2&1&2&2&1&2&1&3\\2&1&4&0&2&2&1&0&1&6\\1&3&2&1&4&0&2&2&0&3\\
    0&1&0&2&2&6&1&4&3&0\\1&1&2&1&0&4&0&2&2&3\\0&0&0&0&4&2&2&0&1&0\\
    1&1&2&1&2&4&1&2&2&3\\2&2&4&3&0&0&0&6&0&6\end{bmatrix}.$$

We find that

$\{y_1=a_1,y_2=a_2,y_3=a_3, y_4=a_4, y_5=a_6\}$ is a basic set and  $$X=\begin{bmatrix}

     1&     2&     2&     1&     2&     2&     1&     2&     1&     3\\
     2&     1&     4&     0&     2&     2&     1&     0&     1&     6\\
     1&     3&     2&     1&     4&     0&     2&     2&     0&     3\\
     0&     1&     0&     2&     2&     6&     1&     4&     3&     0\\
     0&     0&     0&     0&     4&     2&     2&     0&     1&     0\end{bmatrix}$$

is the matrix with rows the vectors $y_i$.  The sum of the vectors $y_i$  is $$y=( 4,  7,     8,     4,    14,    12,     7,     8,     6,    12).$$
We find that $\beta(1)=\beta(3)=\beta(10)=\frac{1}{4}(1,2,1,0,0)$, $\beta(2)=\frac{1}{7}(2,1,3,1,0)$, $\beta(4)=\beta(8)=\frac{1}{4}(1,0,1,2,0)$,
$\beta(5)=\beta(7)=\frac{1}{7}(1,1,2,1,2)$, $\beta(6)=\beta(9)=\frac{1}{6}(1,1,0,3,1)$ and
  $$R(\beta)= \{P_1=\beta(5),P_2=\beta(6),P_3=\beta(4),P_4=\beta(1),P_5=\beta(2)\}.$$ The cardinal number of $R(\beta)$ is equal to $r$. This ensures a rank factorization of $A$ and also that
the subspace $X$ generated by the vectors $y_i$ is  lattice-subspace (especially $X$ is a sublattice).
Hence $Z=X$ is the minimal lattice-subspace which contains the vectors $y_i$ and a positive basis of $Z$
 is given by the formula
$$(b_1,b_2, b_3,b_4,b_5)^T=L^{-1}(y_1,y_2,y_3,y_4,y_5)^T,$$
where $L$ is the $5\times 5$ matrix with columns the vectors
$P_1,P_2,P_3,P_4,P_5$ of $R(\beta)$. We find that
 $$b_1=(0,0,0,0,14,0,7,0,0,0),b_2=(0,0,0,0,0,12,0,0,6,0),$$
$$ b_3=(0,0,0,4,0,0,0,8,0,0), b_4=(4,0,8,0,0,0,0,0,0,12),b_5=(0,7,0,0,0,0,0,0,0,0).$$
is  a positive basis of  $Z$ and
 $\{i_1=5,i_2=6, i_3=4,i_4=1,i_5=2\}$ is  a set of nodes of this basis. Therefore $F$ is a $10\times 5$ matrix with columns
 the fifth  column of $A$ multiplied by $\frac{1}{b_1(5)}=\frac{1}{14}$, the sixth column of $A$ multiplied by $\frac{1}{b_2(6)}=\frac{1}{12}$, the fourth column of $A$ multiplied by $\frac{1}{b_3(4)}=\frac{1}{4}$,  the first column of $A$ multiplied by $\frac{1}{b_4(1)}=\frac{1}{4}$ and  the second column of $A$ multiplied by $\frac{1}{b_5(2)}=\frac{1}{7}$. Therefore

 $$F=\begin{bmatrix}\frac{2}{14}&\frac{2}{12}&\frac{1}{4}&\frac{1}{4}&\frac{2}{7}\\
                               \frac{2}{14}&\frac{2}{12}&\frac{0}{4}&\frac{2}{4}&\frac{1}{7}\\
                                \frac{4}{14}&\frac{0}{12}&\frac{1}{4}&\frac{1}{4}&\frac{3}{7}\\
                               \frac{2}{14}&\frac{6}{12}&\frac{2}{4}&\frac{0}{4}&\frac{1}{7}\\
                            \frac{0}{14}&\frac{4}{12}&\frac{1}{4}&\frac{1}{4}&\frac{1}{7}\\
                            \frac{4}{14}&\frac{2}{12}&\frac{0}{4}&\frac{0}{4}&\frac{0}{7}\\
                            \frac{2}{14}&\frac{4}{12}&\frac{1}{4}&\frac{1}{4}&\frac{1}{7}\\
                            \frac{0}{14}&\frac{0}{12}&\frac{3}{4}&\frac{2}{4}&\frac{2}{7}\end{bmatrix}\;\;,
                             V=\begin{bmatrix}0&0&0&0&14&0&7&0&0&0\\0&0&0&0&0&12&0&0&6&0\\
0&0&0&4&0&0&0&8&0&0\\4&0&8&0&0&0&0&0&0&12\\0&7&0&0&0&0&0&0&0&0
\end{bmatrix}$$
with  $R=A-FV=0$ and  $A=FV$ is a NRF of $A$.}

\end{example}

\begin{example}\label{ex3}{\bf (rank(A)=2)}
{\rm Suppose that $A$ is a $n\times 16$ nonnegative real matrix, $\{y_1,y_2\}$ is a basic set  of the rows of $A$, i.e. $r=2$  and  that
 $$X=\left[\begin{array}{rrrrrrrrrrrrrrrr}1&2&7&3&6&4&5&8&4&2&3&9&4&7&9&1\\
3&6&7&8&6&3&2&3&2&5&6&5&9&8&7&4\end{array}\right],$$ is the matrix with rows the vectors $y_1,y_2$ of the basic set.
The values of $\beta$ are on the one-dimensional simplex $\Delta$ of $\mathbf{R}^2_+$, therefore the convex polytope generated by $R(\beta)$ is the line segment defined by the values of $\beta$ with  minimum  and  maximum first coordinate. We find that $b(16)=\frac{1}{5}(1,4)$, is the value of $\beta$ with minimum  first coordinate and
 $b(8)=\frac{1}{11}(8,3)$  the one with maximum, therefore $b(16)$ and
 $b(8)$ are the vertices of $K$. By Theorem~\ref{Prop5}, a positive basis $\{b_1,b_2\}$ of the minimal lattice-subspace $Z$ which contains $y_1,y_2$ is given by the formula $(b_1,b_2)^T=L^{-1}X$ , where  $L$ is the matrix with columns the vectors $b(16)$,
 $b(8)$. We have  $$(b_1,b_2)^T=L^{-1}X=-\frac{55}{29}\left[\begin{array}{rrrr}\frac{3}{11}&-\frac{8}{11}\\
                                                                     -\frac{4}{5}&\frac{1}{5}\end{array}\right] X= -\frac{1}{29}\left[\begin{array}{rrrr}15&-40\\
                                                                     -44&11\end{array}\right] X . $$
 We find that
  $$b_1=\frac{1}{29}(105,   210,   175,   275,   150,    60,     5,     0,    20,   170,   195,    65,   300,   215,   145,   145),$$
  $$b_2=\frac{1}{29}(11,    22,   231,    44,   198,   143,   198,   319,   154,    33,    66,   341,    77,   220,   319,     0)$$
 and  $\{i_1=16,i_2=8\}$ is a set nodes of the basis. Therefore   $A=FV$ is a NRF of $A$ where $F$ is the  matrix with first column the $16$-column of $A$ multiplied by $\frac{1}{b_1(16)}=\frac{29}{145}$ the second column of $F$ is the $8$-column of $A$ multiplied by $\frac{1}{b_2(8)}=\frac{29}{319}$  and $V$ is the matrix with rows the vectors $b_1,b_2$.}
\end{example}

\begin{example}\label{ex20}\rm{Suppose that

$$ A=\left[\begin{array}{rrrrrrrrrrrrrrrr}
    2  &   1  &   0  &   0 &    0   &  0 &    1  &   2  &   1 &    1  &   2\\
     1  &   2 &    2 &    1  &   0  &   0  &   2  &   1  &   2  &   0  &   0\\
     0  &   0   &  1   &  2  &   2 &    1  &   2  &   1  &   1  &   0  &   1\\
     0  &   0  &   0  &   0  &   1   &  2  &   1  &   2  &   2  &   1   &  3\\
     3   &  3  &   2  &   1  &   0  &   0  &   3  &   3  &   3  &   1  &   2\\
     1  &   2  &   3   &  3  &   2  &   1  &   4   &  2  &   3   &  0   &  1\\
     0  &   0   &  1   &  2   &  3  &   3  &   3   &  3   &  3  &   1  &   4\\
     2  &   1  &   0  &   0   &  1   &  2   &  2   &  4  &   3  &   2  &   5 \end{array}\right].$$

Following the matlab program we find that
$\{y_1=a_1,y_2=a_2,y_3=a_3, y_4=a_4\}$ is a basic set,   the convex polytope $K$ generated by $R(\beta)$ has seven vertices therefore $d=7$ is the indermediate dimension of the factors.  $d-r=3$, therefore three new vectors $y_5,y_6,y_7$ are determined and  a positive basis of the minimal lattice subspace which contains the rows of $A$ is determined. The next factors of are given

$$ F=\left[\begin{array}{rrrrrrrrrrrrrrrr} 0 &        0  &    &    0 &    0.3333  &        0 &    0.2500  &   0.3333\\
         0  &         0  &    0.3333  &    0.6667  &    0.3333    &       0  &    0.1667\\
    0.3333  &    0.6667  &    0.6667   &        0  &    0.1667     &      0   &        0\\
    0.6667  &    0.3333   &        0    &       0    &       0  &    0.2500    &       0\\
         0   &        0  &    0.3333  &    1.0000  &    0.3333  &    0.2500  &    0.5000\\
    0.3333  &    0.6667  &    1.0000  &    0.6667   &   0.5000    &       0  &    0.1667\\
    1.0000  &    1.0000  &    0.6667    &       0  &    0.1667  &    0.2500    &       0\\
    0.6667  &    0.3333    &       0  &    0.3333   &        0  &    0.5000  &    0.3333\end{array}\right],$$

$$ V=\left[\begin{array}{rrrrrrrrrrrrrrrr}
         0     &    0     &    0    &     0     &    0   & 3.0000  &  0.6923  &  1.2688   & 3.0000   &      0  &  1.0000\\
         0     &    0     &    0     &    0   & 3.0000    &     0  &  0.9231 &   0.6918  &       0   &      0  &  1.0000\\
         0     &    0     &    0   & 3.0000    &     0    &     0  &  1.1538  &  0.1159  &       0   &      0   &      0\\
         0   & 3.0000      &   0    &     0    &     0    &     0  &  0.9231  &  0.6918  &  3.0000    &     0   &      0\\
         0  &       0   & 6.0000     &    0    &     0     &    0   & 2.3077  &  0.2317  &       0    &     0    &     0\\
         0    &     0     &    0     &    0    &     0    &     0  &  0.9231  &  3.6940  &       0    &4.0000   & 8.0000\\
    6.0000    &     0    &     0     &    0    &     0    &     0  &  1.3846  &  2.5377   &      0    &     0   &      0\end{array}\right].$$

and the test matrix $R=A-FV=0$, under a very small round off error.}
\end{example}

\begin{example}\label{Camb}{\rm In ~\cite{Camp}, page 180, the nonnegative rank factorization of the matrix
$$A=\left[\begin{array}{rrrrrrrrrrrrrrrr}4&0&4&11\\0&2&2&0\\4&0&4&8\\5&0&5&15\\1&1&2&2\end{array}\right]$$
is determined, as an example of the   algorithm expanded in ~\cite{Camp} for   rank factorization.\\
By applying   Theorem ~\ref{1} we determine also a NRF of $A$  as follows: We find that $\{y_1=a_1,y_2=a_2,y_3=a_3\}$ is a basic set,
 $\beta(1)=(\frac{1}{2},0,\frac{1}{2})$,
$\beta(2)=(0,1,0)$,
$\beta(3)=(\frac{2}{5},\frac{1}{5},\frac{2}{5})$,
$\beta(4)=(\frac{11}{19},0,\frac{8}{19})$ are the values of $\beta$ of the vectors
$y_i$ and that  $\beta(1),\beta(2),\beta(4)$ are the vertices of the convex polytope $K$ generated by the values of $\beta$. Therefore  $r=d=3$ and  the factorization $A=FV$  of Theorem ~\ref{1} is a rank factorization of $A$. A positive basis of $Z$ is given by the formula $(b_1,b_2,b_3)^T=L^{-1}(y_1,y_2,y_3)^T$ where $L$ is the matrix with columns  $\beta(1),\beta(2),\beta(4)$. We have $b_1=(0,2,2,0),b_2=(8,0,8,0), b_3=(0,0,0,19)$ and  $\{i_1=2,i_2=1,i_3=4\}$ is a set of nodes.  Therefore $A=FV$, where the first column of $F$ is the second column of $A$ multiplied by $\frac{1}{2}$, the second  is the first column of $A$ multiplied by $\frac{1}{8}$ and the third  is the fourth  column of $A$ multiplied by $\frac{1}{19}$ and $V$ is the matrix with rows the vectors $b_i$.
So we have
$$F=\begin{bmatrix}
               0&\frac{1}{2}&\frac{11}{19}\\
               1&0&0\\
               0&\frac{1}{2}&\frac{8}{19}\\
               0&\frac{5}{8}&\frac{15}{19}\\
              \frac{1}{2}&\frac{1}{8}&\frac{2}{19}
               \end{bmatrix},\;\;V=\begin{bmatrix}0&    2&    2&     0 \\ 8&     0&     8&    0\\
          0&     0&     0&     19\end{bmatrix}$$
         with   $FV=A$. This factorization is in fact the same with the one of ~\cite{Camp}. }
\end{example}

\subsection{ $A$ contains a  diagonal principal submatrix of the same rank }\label{diag}

Suppose  that   $A$  {\it contains a  diagonal principal submatrix} of the same rank, i.e. that under a set of permutations between rows of $A$  and permutations between columns of $A$ we take a new matrix $\overline{A}$ which has a diagonal $k\times k$ submatrix $D_k$ so that $\rank(A)=\rank(D_k)=k$.\\
   We  apply the algorithmic process for the factorization of the new matrix $\overline{A}$.\\
The rows $y_1,...,y_k$ of $\overline{A}$ corresponding to $D_k$ define a basic set of the rows of $\overline{A}$  and consider the basic function $\beta$ of the vectors $y_1,...,y_k$. Suppose that the $j$-columns of $\overline{A}$, for  $j=\nu,...,\nu+k$, are the columns of $\overline{A}$ corresponding to $D_k$. It is easy  that for any $i=\nu+t$ with  $0\leq t\leq k$ the value $\beta(i)$ of $\beta$ is the $t$-column of $D_k$ divided by its diagonal  element, therefore $\beta(i)$ is the $t$-vertex  $e_t$ of the simplex $\Delta$ of $\mathbf{R}^k_+$. Therefore the values $\beta(i)$ of $\beta$ for $i=\nu,...,\nu+k$, are the vertices $e_1,...e_k$ of the simplex $\Delta$ of $\mathbf{R}^k_+$. Since the values of $\beta$ are on the simplex $\Delta$ of $\mathbf{R}^k_+$, the convex polytope $K$ generated by the vales of $\beta$ is the whole simplex $\Delta$ and the vertices of $K$ are the vertices  $e_1,...,e_k$ of the simplex $\Delta$ of $\mathbf{R}^k_+$. So
 a positive basis of a minimal lattice-subspace $\overline{Z}$ which contains  $y_1,...,y_k$ and therefore also the rows of $\overline{A}$, is given by the formula
 $(b_1,...b_k)^T=L^{-1}(y_1,...,y_k)^T$ where $L$ is the matrix with columns the the vertices $e_1,...,e_k$ of $K$, therefore $L$ is the identical $k\times k$ matrix. This implies that  $\{y_1,...,y_k\}$ is a  positive basis of $\overline{Z}$ and also that  $\overline{Z}$  is the subspace generated by the rows of $\overline{A}$.\\
 The set  of indexes $\{\nu,...,\nu+k\}$ is a set of nodes of the basis $\{y_1,...,y_k\}$ because the columns $\nu,...,\nu+k$ of the matrix $\overline{V}$ with rows the vectors $y_i$ are the columns of the diagonal matrix $D_k$.
 Therefore $\overline{A}=\overline{F}\;\overline{V}$ where $\overline{F}$ is the $n\times k$ matrix so that for $j=1,...,k$ the $j$-column of $\overline{F}$ is the  $\nu+j$-columns of $\overline{A}$ multiplied by $\frac{1}{y_j(\nu+j)}$  and
  $\overline{V}$ is the matrix with rows $y_1,...,y_k$.\\
 In the sequel, by the inverse process, where, by (\ref{fact}), the  permutation of two columns in the initial matrix $A$ implies    the permutation of the corresponding columns in $\overline{V}$  and  the  permutation of two rows in initial matrix $A$ implies  the permutation of the corresponding rows in $\overline{F}$ we take from $\overline{F}$, $\overline{V}$ the new matrices $F$, $V$ which are factors of the initial matrix $A$, i.e. $A=FV$,  with intermediate dimension of $F,V$ equal to $k$.\\
 Therefore  we have proved the existence of a nonnegative rank factorization of $A$. But for the determination of the factors $F,V$ the set of permutations  for a diagonal submatrix $D_k$ is needed.\\
 In the next theorem we prove that the  factorization of $A$  which is given by   Theorem~\ref{1} is a rank factorization of $A$. So by Theorem~\ref{1} we can determine directly a rank factorization of $A$, independently from the knowing of the set of permutations  for a diagonal submatrix.

 \begin{theorem}\label{2} If $A$ is a $\;n\times m$ nonnegative real matrix which contains a  diagonal principal submatrix of the same rank  $k$  with the rank of $A$, then   the factorization  of $A$  determined by  Theorem~\ref{1}, is a nonnegative rank factorization of $A$.
 \end{theorem}

 \begin{proof}  By the above process we have taken a factorization $A=FV$ of $A$ where $V$ is coming from $\overline{V}$ by a set  of  permutations of the columns of $\overline{V}$. Also we have shown that  the rows $y_1,...y_k$  of $\overline{V}$ define a positive basis of a minimal lattice-subspace $\overline{Z}$  which contains the rows of $\overline{A}$ and that $\overline{Z}$ is the subspace generated by the rows of $A$.
 Since $V$ is coming from $\overline{V}$ by a set $\sigma$ of  permutations of the columns of $\overline{V}$ we have that the subspace $W$ of $\mathbf{R}^m$ generated by the rows of $V$ is a lattice-subspace. The proof is the following:\\
  Suppose that $S:\mathbf{R}^m\longrightarrow\mathbf{R}^m$ so that $S(x)$ is the vector of $\mathbf{R}^m$ arising by applying on the coordinated of $x$  the set of permutations $\sigma$. Then $S$ is linear, one-to-one and onto with the property: $x\in \mathbf{R}^m_+$ if and only if $S(x)\in \mathbf{R}^m_+$, i.e. $S$ and $S^{-1}$ are positive. This implies that  rows $b_i=S(y_i)$,  $i=1,...,k$  of $V$,  define a positive basis of $W$ and   $S(\overline{Z})=W$.\\
 Suppose that  $x,y\in W$. Then there exist a unique pair of vectors $\overline{x},\overline{y}\in\overline{Z}$ with  $x=S(\overline{x})$ and $y=S(\overline{y})$. Since $\overline{Z}$ is a lattice-subspace there exists $\overline{z}\in\overline{Z}$ which is the supremum of $\{\overline{x},\overline{y}\}$ in $\overline{Z}$, i.e. $\overline{z}$ is the minimum of all upper bounds of $\{\overline{x},\overline{y}\}$ which belong to $\overline{Z}$. If $z=S(\overline{z})$, because of the positivity of $S$ and $S^{-1}$ we have that   $z$  is the minimum of all upper bounds of $\{x,y\}$ which belong to $W$, therefore $z$ is the supremum of $\{x,y\}$ in $W$ and $W$ is a lattice-subspace. Also $\dim(W)=k$. Since $A=FV$ and by  the relation  (~\ref{fact}) we have that the rows of $A$ belong to $W$, therefore $W$ is the subspace of $\mathbf{R}^m$ generated by the rows of $A$, because $\dim(W)=k=\rank(A)$. This shows that the minimal lattice-subspace which contains the rows of $A$ is unique and equal to $W$.
 Suppose that   $A=\mathbb{F}\mathbb{V}$ is the factorization of $A$ of Theorem~\ref{1}. Then the rows of  $\mathbb{V}$ generate a minimal lattice-subspace $\mathbb{Z}$ which contains the rows of $A$, therefore $\mathbb{Z}=W$ because this minimal lattice-subspace is unique. So we have that $k=\rank(A)=\dim(\mathbb{Z})$  is the intermediate dimension of $\mathbb{F},\mathbb{V}$ and therefore  the factorization  $A=\mathbb{F}\mathbb{V}$ of Theorem~\ref{1} is a rank factorization of $A$.
  \end{proof}
 The next  is an example of the above process.
\begin{example}\label{exkal}\rm{ The matrix $A$ below is the one of  Example 7, Section 7 of  ~\cite{Kal}, where
an exact, symmetric nonnegative rank factorization of $A$, $A=WW^T$ is determined. Below we take a factorization of $A$ by two ways. In the first way, by a set of permutations we find a diagonal submatrix and in the second one we apply directly Theorem ~\ref{1}, but of course our factorization is not symmetric.\\
 $$A=\begin{bmatrix}13&15&12&0&10&14\\15&25&0&0&0&20\\12&0&37&4&30&9\\0&0&4&16&0&12\\10&0&30&0&25&5\\14&20&9&12&5&26\end{bmatrix}.$$
By permuting  third and forth row,   fourth and fifth row and   second and  third column of $A$   we take the matrix
$$\overline{A}=\begin{bmatrix}13&12&15&0&10&14\\15&0&25&0&0&20\\0&4&0&16&0&12\\10&30&0&0&25&5\\12&37&0&4&30&9\\14&9&20&12&5&26\end{bmatrix}.$$
  with a diagonal  submatrix $D_3$.
A basic set $\{y_1,y_2,y_3\}$ of the rows of $\overline{A}$ is consisting by the rows of the diagonal submatrix $D_3$ i.e. by the second, third and fourth row of   $\overline{A}$. We take the basic function $\beta$ of the vectors $y_i$.  The values of $\beta$  corresponding  to the  columns of  $D_3$ are the vertices of the simplex $\Delta$ of $\mathbf{R}^3_+$. Indeed, we have   $\beta(3)=(1,0,0)$, $\beta(4)=(0,1,0)$ and $\beta(5)=(0,0,1)$ and  any other value of $\beta$ is on $\Delta$. So the convex polytope $K$ generated by $R(\beta)$ is the simplex $\Delta$ and   a positive basis $\{b_1,b_2,b_3\}$ of a minimal lattice-subspace $Z$ which contains the vectors $y_i$ is given by the formula
$(b_1,b_2,b_3)=L^{-1}(y_1,y_2,y_3)^T$ where $L$ is the matrix with  columns the vertices of $K$, hence $L$ is the identical matrix  $I_3$. So we have that $\{y_1,y_2,y_3\}$ is
 a positive basis of the minimal lattice-subspace $\overline{Z}$ which contains the vectors $y_i$ and therefore also the rows of $\overline{A}$ and     $\{i_1=3,i_2=4,i_4=5\}$ is a set of nodes of the basis.
Therefore $\overline{A}=\overline{F}\;\overline{V}$, where $\overline{F}$, according to Theorem~\ref{1}  is the matrix consisting by the third, fourth and fifth column of $\overline{A}$ multiplied by $\frac{1}{25}$, $\frac{1}{16}$ and $\frac{1}{25}$, respectively and $\overline{V}$ the matrix with rows the vectors $y_1,y_2,y_3$. By the inverse process  we permute   second  and third column  of $\overline{V}$ and we permute fourth and fifth  and  third  and  fourth row of $\overline{F}$. So we  find that  $F=\begin{bmatrix}3/5&0&2/5\\1&0&0\\0&1/4&6/5\\0&1&0\\0&0&1\\4/5&3/4&1/5\end{bmatrix}$,
$V=\begin{bmatrix}15&25&0&0&0&20\\0&0&4&16&0&12\\10&0&30&0&25&5\end{bmatrix}$  and we have $FV=A$.

 Following Theorem~\ref{1}  we find a  rank factorization of $A$ as follows:
We find that $\{y_1=a_1,y_2=a_2,y_3=a_3\}$ is a basic set of the rows of  $A$ and
$$X=\begin{bmatrix}13&15&12&0&10&14\\15&25&0&0&0&20\\12&0&37&4&30&9\end{bmatrix},$$ is the matrix of the vectors $y_i$ and
 $y=(40,    40,    49,     4,    40,    43)$ is the sum of these vectors. By our matlab program we find that
$$G =\begin{bmatrix}
    0.3250&    0.3750&    0.2449&        0.0000&  0.2500&    0.3256\\
    0.3750&    0.6250&    0.0000&        0.0000&        0.0000&    0.4651\\
    0.3000&    0.0000&    0.7551&    1.0000&    0.7500&    0.2093\end{bmatrix},$$
is the matrix with columns the values $\beta(i)$ of $\beta$ and that

$$K1=\begin{bmatrix} 0.0000&   0.2500&    0.3750\\
         0.0000&       0.0000&   0.6250\\
    1.0000&    0.7500&         0.0000\end{bmatrix}$$
is the matrix with columns the vertices of the convex polytope $K$ generated by the values of $\beta$.
    Note that $d=r=3$ therefore the subspace $Z$ generated by the rows of $A$, is the minimal lattice-subspace which contains the vectors $y_1,y_2,y_3$ and the rows of the matrix $U=L^{-1}X$ are the vectors of a positive  basis of $Z$, where $L$ is the matrix with columns the vertices of $K$, i.e. $L=K1$.
    We find  that
    $$U= \begin{bmatrix}
     0&    0& 1&4& 0& 3\\
    16&0&48&0&40&8\\
    24&40&0&0&0&32\end{bmatrix},$$
therefore $\{i_1=4,i_2=5,i_3=2\}$ is a set of nodes of the basis. Therefore $A=FV$, where
the first column  of $F$ is the fourth column of $A$ multiplied by $\frac{1}{4}$, the second  column  of $F$ is the fifth column  of $A$ multiplied by $\frac{1}{40}$ and third column  of $F$ is the second column  of $A$ multiplied by $\frac{1}{40}$ and $V=U$.  Therefore
$$F=\begin{bmatrix}0&1/4&3/8\\0&0&5/8\\1&3/4&0\\4&0&0\\0&5/8&0\\3&1/8&1/2\end{bmatrix},$$
 and  it is easy to check that
$A=FV$.

}
\end{example}

\section{Appendix}

\subsection{Lattice-subspaces and positive bases in $\mathbb{R}^m $}

We present here the  basic mathematical notions and results of
~\cite{POLY96} and ~\cite{POLY99} which are needed for this
article. In these articles finite dimensional lattice-subspaces or
equivalently, finite dimensional ordered subspaces with positive
bases of the space $E=C(\Omega) $ of the real valued functions
defined on a compact Hausdorff topological space $\Omega$ are
studied.
For compatibility with our article we present these results in the case where
$\Omega=\{1,2,...,m\}$ and   $$E=\mathbb{R}^m
=\{x=(x(i))\;|\;\text{with}\;x(i)\in \mathbb{R},\;\text{for
any}\;i=1,2,...,m \}.$$
 The space  $\mathbb{R}^m $ is ordered by the pointwise ordering i.e. for any
$x, y\in \mathbb{R}^m $ we have $x\geq y$ if and only if $x(i)\geq
y(i)$ for each $i$. Then  $$\mathbb{R}^m_+=\{x\in
\mathbb{R}^m\big |x(i)\geq 0\;\text{for each} \; i\},$$ is
the positive cone of $\mathbb{R}^m$. For any $x,y\in \mathbb{R}^m$, $x\vee y=z$ where
$z(i)=x(i)\vee y(i)$ is the supremum of $\{x,y\}$ and $x\wedge
y=w$ where $w(i)=x(i)\wedge y(i)$ is the infimum of $\{x,y\}$. For real numbers $a,b$, $a\vee b$ is the maximum and $a\wedge
b$ is the minimum of $\{a,b\}$.
 Any subspace $X$  of $\mathbb{R}^m$, ordered by
the induced ordering, is an {\it ordered subspace} of
$\mathbb{R}^m $. Then $X_+=X \cap \mathbb{R}^m_+ $ is  the
positive cone of $X$ and for any $x,y\in X$ we have $x\geq
y\Longleftrightarrow x-y\in X_+$.\\
Suppose that $X$ is an ordered subspace of  $\mathbb{R}^m$. $X$
  is a {\it sublattice} or a Riesz subspace of $\mathbb{R}^m$ if for every $x,y \in X$, $x
\vee y$ and $ x \wedge y$ belong to  $X$.\\
Suppose that
$x,y\in X$. If a vector $z\in X$ exists so that $z$ is the minimum
of the upper bounds of $\{x,y\}$ which belong to $X$, then $z$ is
the supremum of  $\{x,y\}$ in $X$ and we write $z=\sup_X\{x,y\}$.
 Similarly  the maximum  of the lower
bounds of $\{x,y\}$ in $X$ (if exists) is the infimum,
$\inf_{X}\{x,y\}$, of $\{x,y\}$ in $X$. If for any
  $x,y\in X$ the supremum $\sup_X\{x,y\}$ and
 the infimum $\inf_X\{x,y\}$ of $\{x,y\}$ in
$X$ exist,  we say that
 $X$ is a {\it lattice-subspace} of $\mathbb{R}^m$.
  Then we have
$$\inf_X\{x,y\}\leq x \wedge y \leq x \vee y\leq \sup_X\{x,y\}.$$
Any sublattice is a lattice-subspace but the converse is not true.
  The set  $\{b_1,b_2,...,b_r\}$ is a {\it positive
basis} of $X$, if $\{b_1,b_2,...,b_r\}$ is a basis of $X$ and $$X_+=\{x= \sum_{i=1}^r
\lambda_i b_i
  \mid \lambda_i \geq 0 \text{ for each $i$}\},$$ i.e. the positive cone of $X$ is the set of vectors of $X$
  with nonnegative coordinates in the basis $\{b_1,b_2,...,b_r\}$. Then
  for any $x=\sum_{i=1}^r \lambda_i b_i \in X$ we have $$x\geq 0\Longleftrightarrow\lambda_i \geq 0,\text{ for any}\; i.$$
 Although $X$ has infinitely many bases, the existence
of a positive basis of $X$ is not always ensured.  The next result
see in ~\cite{POLY96},  identifies    the class of ordered
subspaces of $\mathbb{R}^m$ with a positive basis with the class
of the lattice-subspaces of $\mathbb{R}^m$. This result is in fact
the Choquet-Kendall theorem for finite-dimensional ordered
subspaces but there is also a more elementary proof based on the
theory of ordered spaces.

\begin{theorem} An ordered subspace $X$ of $\mathbb{R}^m$ is a lattice-subspace of
$\mathbb{R}^m$ if  and only if $X$  has a
positive basis.
\end{theorem}
Suppose that   $A$ is a nonempty subset of $\mathbb{R}^m_+$. The
intersection $S(A)$ of all sublattices of $\mathbb{R}^m$ which
contain $A$ is again a sublattice of $\mathbb{R}^m$. Then $S(A)$ is the smallest (the minimum) sublattice of $\mathbb{R}^m$ which contains $A$  and is referred  as  {\it the sublattice} of $\mathbb{R}^m$ {\it generated
by $A$}.
 In the case of lattice-subspaces the intersection of all lattice-subspaces which contain $A$ is not necessarily a lattice-subspace. The reason is that in  lattice-subspaces,  the lattice operations  are not the induced ones by  $\mathbb{R}^m$ but are depenting on the subspaces.  So in ~\cite{POLY99} the notion of the  minimal lattice-subspace is defined as follows:  $Y$ is  {\it  minimal lattice-subspace} of $\mathbb{R}^m$ which contains $A$ if $Y$ is a lattice-subspace of $\mathbb{R}^m$, $A\subseteq Y$ and  does not exist a proper subspace of $W$ of $Y$ which is a lattice-subspace of $\mathbb{R}^m_+$  which contains $A$.
 As it is shown in ~\cite{POLY99}, Example 3.21,  even in the case of $\mathbb{R}^m$,  a minimal lattice-subspace which contains $A$ is not necessarily unique and also that such a subspace  is not necessarily contained in the sublattice of $\mathbb{R}^m$ generated by $A$.

 In ~\cite{POLY96} and ~\cite{POLY99} it is supposed that  $z_1,z_2,...,z_r$  are  fixed, linearly
independent, positive vectors of $C(\Omega)$ and
$X=[z_1,z_2,...,z_r]$ is the subspace of $C(\Omega)$ generated by
these vectors and the next problems are studied:

 $(i)$ Is $X$ is a lattice-subspace or a
sublattice  of  $C(\Omega)$?

$(ii)$ Does   a finite-dimensional minimal lattice-subspace $Z$ of $C(\Omega)$ which contains $X$ exist?  Determine $Z$ whenever  exists.

In all these cases the problem is connected with the determination
of a positive basis of the ordered subspaces $X$ or $Z$ and  the results and also the determination of a positive
basis   are based on the study of the generating vectors $z_1,z_2,...,z_r$.
The next function has been defined in ~\cite{POLY96} and is
crucial in  this theory. The function
$$ \beta(i)=\Bigl(\frac{z_1(i)}{z(i)},\frac{z_2(i)}{z(i)},...,
\frac{z_r(i)}{z(i)}\Bigr),\;\text{for each}\;\;
i=1,2,...,m\;\text{with\;} z(i)>0,$$  where $z=z_1+z_2+...+z_r$,
is  the {\it  basic function }\footnote{In ~\cite{POLY96}, the basic function $\beta$ is referred as the basic curve.} of $z_1,z_2,...,z_r$.
The set
 $$R(\beta)= \{\beta(i)\;|\;i=1,2,...,m\;\text{with\;} z(i)>0\},$$ is the range of $ \beta$ where $R(\beta)$, as a set, is consisting by mutually different vectors. The next Lemma is obvious because the rank of the matrix $M$ with rows the vectors $z_i$ is equal to $r$  and   the values of $\beta$ are positive multiples of the columns of $M$.
 \begin{lemma}\label{l1}The rank of the matrix with columns the vectors of $R(\beta)$ is equal to $r$.
  \end{lemma}
Denote by  $K$ be the convex polytope  of $\mathbb{R}^m$ generated by the finite set  $R(\beta)$ (the convex hull of $R(\beta)$) and suppose that $P_1,P_2,...,P_d$ are the vertices (extreme points) of $K$. Since $R(\beta)$ is finite, the vertices of $K$ are vectors of $R(\beta)$, therefore any $P_k$ is the image of an index $i_k$, i.e. $P_k=\beta(i_k)$. We will denote below by $(a_1,a_2,...,a_k)^T,$ where $a_1,a_2,...,a_k$ are vectors of $\mathbb{R}^m$,  the $k\times m$  matrix with rows the vectors $a_i$. Also in the next results,  $z_1,z_2,...,z_r$ are linearly independent, positive vectors of $\mathbb{R}^m$, $X=[z_1,...,z_r]$,   $\beta$ is the basic function of the vectors $z_i$ and $K$ the convex polytope generated by  $R(\beta)$.

\begin{theorem}[~\cite{POLY99}, Theorem 3.6]\label{Prop3}
The subspace $X$ of $\mathbb{R}^m$ generated by the vectors
$z_1,z_2,...,z_r$,  is a sublattice of $\mathbb{R}^m$ if
and only if $R(\beta)$ has exactly $r$ elements.

If  $R(\beta)= \{P_1,P_2,\ldots,P_r\}$, then a  positive basis
$\{b_1,b_2,...,b_r\}$ of $X$ is given by the formula:
\begin{equation}\label{b1}
(b_1,b_2,...,b_r)^T=L^{-1}(z_1,z_2,...,z_r)^T,
\end{equation}
where $L$ is  the $r\times r$ matrix with  columns the vectors
$P_1,P_2,...,P_r$.

\end{theorem}

\begin{theorem}[~\cite{POLY96}, Theorem 3.6]\label{Prop4}
The subspace $X$ of $\mathbb{R}^m$ generated by the vectors
$z_1,z_2,...,z_r$,  is a lattice-subspace  of $\mathbb{R}^m$ if
and only if  $K$  is a polytope with $r$ vertices.

Then  a  positive basis
$\{b_1,b_2,...,b_r\}$ of $X$ is given by the formula:
\begin{equation}\label{b1}
(b_1,b_2,...,b_r)^T=L^{-1}(z_1,z_2,...,z_r)^T,
\end{equation}
where $L$ is  the $r\times r$ matrix with  columns the vertices
$P_1,P_2,...,P_r$ of $K$.
\end{theorem}

\begin{theorem}[~\cite{POLY99}, Theorem 3.10]\label{Prop5}

If  $K$ is a polytope with $d$ vertices $P_{1},P_{2},...,P_{d}$, then
a $d$-dimensional  minimal lattice-subspace $Z$ of $\mathbb{R}^m$ which contains the vectors $z_1,z_2,...,z_r$  is constructed as follows:
\begin{enumerate}
\item [(a)] Reenumerate   $R(\beta)$ so that its $r$ first vectors to
be linearly independent and we denote again by   $P_{1},P_{2},...,P_{d}$ the new enumeration.

\item [(b)] Define $d-r$ new  vectors $z_{r+k}$, $k=1,2,...,d-r$ following  the next steps:
First, for any $i=1,2,...m$, we expand  the vector $\beta(i)$ as  convex combination  of the vertices $P_1,P_2,...,P_d$ and suppose that $$\beta(i)=\sum_{j=1}^d\xi_j(i)P_j,$$ is such an expansion of $\beta(i)$.\footnote{Of course $\xi_j(i)\geq 0$ for any $j$ and $\sum_{j=1}^d\xi_j(i)=1.$}
In the sequel for  any $k=r+1,...,d$, we define  the vector $z_k$ of  $\mathbb{R}^m_+$ so that  $$z_{k}(i)=\xi_k(i)z(i),\;\text{for any}\;i=1,2,...,m,$$ where $z$ is the sum of the vectors $z_1,...,z_r$ and the subspace
$$ Z=[z_1,...,z_r,...,z_d]$$ generated by the vectors $z_1,...,z_r,...,z_d$  is a  minimal lattice-subspace  which contains   the vectors $z_1,...,z_r$.  A positive basis $\{b_1,b_2,...,b_d\}$ of $Z$ is is given by the formula:
$$(b_1,b_2,...,b_d)^T=L^{-1}(z_1,...,z_r,...,z_d)^T,$$
where $L$ is the matrix with columns the vertices $R_i$, $i=1,2,...,d$ where $R_1,R_2,...,R_d$ are the vertices of the convex polytope generated by the range $R(\gamma)$ of $\gamma$.\\
Especially the vectors $R_i$ are also given by the vectors $P_1,P_2,...P_d$ of the vertices of $K$ where the first $r$ of them are linearly independent as follows:
$R_i=(P_i,0)$ for any $i=1,2,...,r$ and $R_{r+i}=(P_{r+i},e_i)$ for any $i=1,2,...d-r$, where in the vectors $(P_i,0)$, $0$ is the zero vector of $\mathbb{R}^{d-r}$ and in the vectors $(P_i,e_k)$, $e_k$ is the $k$-component of the usual basis $\{e_1,...,e_{d-r}\}$ of $\mathbb{R}^{d-r}$.
\end{enumerate}

\end{theorem}

\subsection{The Matlab program}
We give below the Matlab program for the determination of the factors $F,V$ of a nonnegative matrix $A1$ .
The input is the initial matrix $A1$. By the function[A,y,index]=Zero(A1) we take the matrix $A$ which is $A1$ without zero columns,  the  sum  $y$  of the rows of $A1$ (in the places with $y(i)=0$, the matrix $A1$ has zero columns) and an index "index" with values 0 and 1.  index=1, if  we have deleted  zero columns from $A1$. The factors  $F,V$   and the test matrix $R=A1-FV$ are given by   function[F,V,R]. If $A1$ has zero columns it is displayed "the matrix has zero columns". Note  that a matlab program for the algorithm determining a positive basis of the lattice-subspace of Theorem~\ref{Prop5} has been also given  in ~\cite{POLY12} but the corresponding part of our matlab program of the factorization of $A$ is not the same, with some variations and improvements.

\begin{verbatim}
a1=min(size(A1))
[A,y,index]=Zero(A1)
[X1]=MaximalLinInd(A')
X=X1'
[N,M]=size(X)
if N==M
display('trivial factorization')
pause
end
z=X(1,:)
for i=2:N
    z=z+X(i,:)
end
for i=1:M
    G(:,i)=X (:,i)/z(i)
 end
L=unique(G','rows')
D=L'
[f,g]=size(D)
if f==2
    a=min(D(1,:))
b=max(D(1,:))
B=[a,b;1-a,1-b]
U=B\X
[F,V,R]=factors(U,A)
if index==1
 [F,V,R]=addzeros(A1,F,V,y)
    display('the matrix has zero columns')
end
  display('the rows define a two-dimensonal lattice-subspace')
   pause
end
%the case of sublattice
if g==N
U=D\X
 [F,V,R]=factors(U,A)
if index==1
 [F,V,R]=addzeros(A1,F,V,y)
    display('the matrix has zero columns')
end
display('Rank factorization, the rows of the matrix generate a sublattice')
pause
end
%the next part determines the vertices of the convex hull of columns of D
B=D';
utrans=bsxfun(@minus,B,B(1,:))
rot=orth(utrans')
uproj=utrans*rot
K=convhulln(uproj)
m=unique (K(:))
K1=B(m,:)'
%the columns of K1 are the vertices of the convex polytope
[N1,M1]=size(K1)
if M==M1
    display('trivial factorization')
    pause
end
if N==M1
    U=K1\X
 [F,V,R]=factors(U,A)
 if index==1
   [F,V,R]=addzeros(A1,F,V,y)
    display('the matrix has zero columns')
 end
 display('Rank factorization, the rows of the matrix generate  a lattice-subspace')
pause
end
if M1>=a1
    display('the intermediate factor is equal to M1')
    pause
end
[C,d]=Reordering(K1)
[H]=secondpartbeta(G,C)
[H2]=Hsecond(G,C,H)
[Y]=LastVectors(X,C,z,H2)
[U]=basismatrix(X,Y,C)
[F,V,R]=factors(U,A)
if index==1
 [F,V,R]=addzeros(A1,F,V,y)
    display('the matrix has zero columns')
end

%%%%%%%%%%%%%%%%%%%%%%%%%%%%%%%%%%%%%%%%%

function[A,y,index]=Zero(A1)
[N,M]=size(A1);
y=A1(1,:);
for i=2:N
    y=y+A1(i,:);
end
A=A1
index=0
r=0
for i=1:M
    if y(i)==0
        j=i-r
       A(:,j)=[];
       index=1
        r=r+1;
    end
 end
end
%%%%%%%%%%%%%%%%%%%%%%%%%%%%%%%%%%%%%%%%%%%%%%%%%%%%%%%%%%%%%%%%%%
% D is a matrix with columns
% a maximal set of linearly independent columns  of $X$.

function[D]=MaximalLinInd(X)
[N,M]=size(X);

c=min(size(X))
D=X(:,1)

r=2;
for i=2:M;
    D=[D X(:,i)]
    [n,m]=size(D)
    if rank(D)<r
        D(:,m)=[]

    else
        if c==min(size(D))
            display (D)
            break
        else
          r=r+1
        end
    end

%%%%%%%%%%%%%%%%%%%%%%%%%%%%%%%%%%%%%%%%%
% The  input E is a nonnegative real  matrix (for example E=A)
% and  U  a matrix  with rows the vectors of a
% positive basis of a minimal lattice-subspace Z which contains the rows of E.
% The output are the factors F,V  of E (E=FV) and the test matrix R=E-FV.

function[F,V,R]=factors(U,E)
 [n,m]=size(U)

for i=1:n
 for j=1:m
 if abs(U(i,j))<1e-6
     W(i,j)=0
     W1(i,j)=0
 else
     W(i,j)=U(i,j)
     W1(i,j)=1

end
end
end

B=eye(n)
for i=1:n
       for j=1:m
        if B(:,i)==W1(:,j)
          k(1,i)=j

        end
    end
end
for i=1:n
   ai=1/U(i,k(i))
F(:,i)=E(:,k(i))*ai
end
V=W
R=E-F*V
end

%%%%%%%%%%%%%%%%%%%%%%%%%%%%%%%%%%%%%%%%%%%%%

% The function reorders the columns of X in a new matrix C whose the d first
% columns are linearly  independent  where  d is the maximum
% number of linearly independent columns of X.

function[C,d]=Reordering(X)
[N,M]=size(X)
D=X(:,1)
W=X(:,1)
r=2;
for i=2:M
    D=[D X(:,i)]
    [n,m]=size(D)
    if rank(D)<r
        D(:,m)=[]
        W=[W X(:,i)]
    else
        r=r+1
    end
end
[d,e]=size(D)
W(:,1)=[]
C=D
[a,b]=size(W)
for i=1:b
    C=[C, W(:,i)]
end
end

%%%%%%%%%%%%%%%%%%%%%%%%%%%%%%%%%%%%%%%%%%%%%%%%%%%%%%%%
% In the input, G is a NxM matrix and D a Nxm matrix with m\leq M and
% the output H is a NxM matrix with H(i,j)=1 if G(:,i)=D(:,j), else H(i,j)=0.
% In the case where $G$ is the matrix with columns the values of \beta
% and D=C is the matrix with columns the vertices of K, then H(i,j)=1
% means that \beta(i)=C(:,j) or equivalently that \beta(i)=P_j, where
% P_j is the j-vertex of the convex polytope K.
% H(i,j)=0 for any j, means that $\beta(i) is not a vertex of K.

function[H]=secondpartbeta(G,D)
[N,M]=size(G)
[n,m]=size(D)
for i=1:M
    for j=1:m
        if G(:,i)==D(:,j)
            H(i,j)=1
        else
            H(i,j)=0
        end
    end
end
end

%%%%%%%%%%%%%%%%%%%%%%%%%%%%%%%%%%%%%%%%%%%%%%%%%%%%%%%%%%%%
% In the next function H is the output of the previous function (function[H])
% and z is the sum of the columns of H.  z(i)=0
% implies   H(i,j)=0 for any j and \beta(i) is not a vertex of K.
% Below, for any i with z(i)=0 we expand \beta(i) as a convex combination
% of the columns of C,  and we put these coefficients in the corresponding
% places of $H$. The new matrix H, is the matrix H2.

function[H2]=Hsecond(G,C,H)
[a,b]=size(C)
f=zeros(b,1)
[a1,b1]=size(H)
z=H(:,1)
for i=2:b1
    z=z+H(:,i)
end
w=ones(1,b)
H2=H
e=[1]
for i=1:a1
    if z(i)==0
        Aeq=[C;w]
        beq=[G(:,i);e]
        lb=zeros(b,1)
        x=linprog(f,[],[],Aeq,beq,lb)
        H2(i,:)=x'
    end
end

%%%%%%%%%%%%%%%%%%%%%%%%%%%%%%%%%%%%%%%%%%%%%%%%%%%%%%%%%%%
%The next function determines the new vectors y_i.

function[Y]=LastVectors(X,C,z,H2)
[N,M]=size(X);
[a,d]=size(C);
n=d-N ;
Y=zeros(n,M);
for k=1:n
    for j=1:M
        Y(k,j)=H2(j,a+k)*z(j);
    end
end
end

%%%%%%%%%%%%%%%%%%%%%%%%%%%%%%%%%%%%%%%%%%%%%%%%%%%%%%%%%%
% Determines a matrix U with rows the vectors of the
% positive basis of the minimal lattice-subspace Z.

function[U]=basismatrix(X,Y,C)
V=[X;Y]
[a,b]=size(C)
Z=zeros(b-a,b)
D=[C;Z]
for i=a+1:b
    D(i,i)=1
end

for i=a+1:b
    D(:,i)=D(:,i)/2
end
U=D\V
end

%%%%%%%%%%%%%%%%%%%%%%%%%%%%%%%%%%%%%%%%%%%%%%%%%%%
function[F,V,R] = addzeros(A1,F,V,y)
m1=length(y)
r=0;
for i=1:m1
    if y(i)>0
        j=i-r
        W(:,i)=V(:,j)
    else
        W(:,i)=A1(:,i)
        r=r+1
    end
end
F=F
V=W
R=A1-F*V
end

\end{verbatim}

\end{document}